\documentclass[amstex,12pt]{article}
\usepackage [latin1]{inputenc}
\usepackage{indentfirst}
\usepackage{amsmath}
\usepackage{amssymb,amsmath,amsthm}
\newtheorem{theorem}{Theorem}[section]

\newtheorem{lemma}{Lemma}[section]
\newtheorem{definition}{Definition}[section]
\newtheorem{remark}{Remark}[section]

\newtheorem{example}{Example}[section]

\newcommand{\beq}{\begin{equation}}
\newcommand{\eeq}{\end{equation}}
\newcommand{\beqn}{\begin{eqnarray}}
\newcommand{\eeqn}{\end{eqnarray}}

\allowdisplaybreaks

\begin{document}
\allowdisplaybreaks

\title{Besicovitch almost periodic solutions of abstract semi-linear differential equations with  delay}
\author {Yongkun Li$^a$\thanks{The corresponding author. Email: yklie@ynu.edu.cn.}, Mei Huang$^a$  and Bing Li$^b$\\
$^a$Department of Mathematics, Yunnan University\\
Kunming, Yunnan 650091, China\\
 $^b$School of Mathematics and Computer Science, Yunnan Minzu University\\
Kunming, Yunnan 650500,  China}
\date{}
\maketitle{}

\begin{abstract}
In this paper, we first use the Bohr property to give a definition of  Besicovitch almost periodic functions, and study some basic properties of Besicovitch almost periodic functions, including the equivalence of the Bohr property and the Bochner property. Then, as an application, we use the contraction principal to obtain the existence and uniqueness of Besicovitch almost periodic solutions for a class of abstract semi-linear differential equations with delay.
\end{abstract}

\section{Introduction}
\setcounter{equation}{0}

As we all know, the almost periodic concept was introduced into mathematics by the Danish mathematician Harald Bohr [1-3]. Bohr almost periodic theory soon attracted the attention of many famous mathematicians at that time, such as W. Stepanov, H. Weyl, N. Wiener, S. Bochner and so on. These mathematicians proposed many important generalizations and variants of Bohr almost periodic concept. We have seen that the generalization of almost periodic functions can be considered from two different perspectives, one is the further structural extension of pure periodicity, and the other is the convergence of trigonometric polynomial sequences in a more general sense than uniform convergence. The second direction of generalization was adopted by A.S. Besicovitch  [4-12].

On the one hand, the concept of almost periodic functions in Besicovich's sense is a natural generalization of the concept of almost periodic functions in Bohr's sense. In the sense of Besicovitch, the space of almost periodic functions  is the completion of  trigonometric polynomials of the form
$$a_{1}e^{i\lambda_{1}t}+a_{2}e^{i\lambda_{2}t}+\ldots+a_{n}e^{i\lambda_{n}t}$$
with respect to the seminorm
$$\|f\|_M=\left(\limsup\limits_{T\rightarrow +\infty}\frac{1}{2T}\int_{-T}^{T}\|f(t)\|_{\mathbb{X}}^{\flat}dt\right)^{\frac{1}{\flat}},$$
where $1\leq\flat<\infty$.
The Besicovitch almost periodic functions  defined in this way (denote by $B^\flat$ the collection of all such functions) have both Bohr and Bochner properties.
 C. Corduneanu and A.S. Brsicovitch presented some basic properties of almost periodic functions in the sense of Besicovitch in \cite{2,4}.  Besicovitch almost periodic functions can also be defined in Marcinkiewicz space using Bohr property or Bochner property. But the space composed of the Besicovitch almost periodic functions  defined by Bohr property or Bochner property is only the proper subspace of  $B^\flat$.
 However, the Besicovitch almost periodic functions defined in either way are not  concrete functions, but   equivalent classes of functions. This makes it difficult to study Besicovitch almost periodic solutions for differential equations. At present, the results of Besicovitch almost periodic solutions of differential equations are few.

 On the other hand, the existence of almost periodic solutions in various senses of differential equations is an important content in the qualitative theoretical research and application research of differential equations, see Refs. [13-22] and the references therein.   But the results of the existence of Besicovitch almost periodic solutions of semi-linear differential equations are still very rare.

 Motivated by the above discussion, in this paper, we mainly study some basic properties of Besicovitch almost periodic functions considering the first direction of generalization, that is to say, we use the Bohr property to give a definition of Besicovitch almost periodic functions, and to study some basic properties of this type of Besicovitch almost periodic functions. Based on these properties and the Banach fixed point theorem to study the existence and uniqueness of Besicovitch almost periodic solutions for   abstract semi-linear differential equations with delay.

\section{Preliminaries}
\setcounter{equation}{0}

In this section, we recall some basic definitions and lemmas of Bohr almost periodic functions, measurable and essentially bounded functions, which are used throughout this paper.\par
Let $(\mathbb{X},\|\cdot\|_{\mathbb{X}})$ be a Banach space, $C(\mathbb{R},\mathbb{X})$ be the set of all continuous functions from $\mathbb{R}$ to $\mathbb{X}$ and $BC(\mathbb{R},\mathbb{X})$ be the space of all bounded continuous functions from $\mathbb{R}$ to $\mathbb{X}$.

For $1\leq\flat<\infty$, let $L_{loc}^{\flat}(\mathbb{R},\mathbb{X})$ be the space of all measurable functions from $\mathbb{R}$ into $\mathbb{X}$ that are  $\flat$-th power locally integrable and $M_{\flat}(\mathbb{R},\mathbb{X})$ be the collection of all $f\in L_{loc}^{\flat}(\mathbb{R},\mathbb{X})$ satisfying $$\|f\|_{M_{\flat}}:=\left(\limsup\limits_{T\rightarrow +\infty}\frac{1}{2T}\int_{-T}^{T}\|f(t)\|_{\mathbb{X}}^{\flat}dt\right)^{\frac{1}{\flat}}<+\infty.$$

We denote by $L^{\infty}(\mathbb{R},\mathbb{X})$   the set of all functions $f:\mathbb{R}\rightarrow \mathbb{X}$ that are measurable and essentially bounded.
 The space $L^{\infty}(\mathbb{R},\mathbb{X})$ is a Banach space with the norm
$$\|f\|_{\infty}:=\inf\left\{D\geq 0:\|f(t)\|_{\mathbb{X}}\leq D~a.e.~t\in\mathbb{R}\right\}.$$

\begin{definition}(\cite{4}) Let $f:\mathbb{R}\rightarrow \mathbb{X}$ be a continuous function, then $f$ is called (Bohr) almost periodic if for each $\varepsilon>0$, there exists $l=l(\varepsilon)>0$ such that in every interval of length $l$ of $\mathbb{R}$, one can find a number $\tau\in(a,a+l)$ with the property (Bohr property)
\begin{align}\label{2.1}
\sup\limits_{t\in \mathbb{R}}\|f(t+\tau)-f(t)\|_{\mathbb{X}}<\varepsilon.
\end{align}The collection of such Bohr almost periodic functions will be denoted by $AP(\mathbb{R},\mathbb{X})$.
\end{definition}

\begin{lemma}(\cite{5}) The space $AP(\mathbb{R},\mathbb{X})$ is a Banach space  when it is endowed with the norm $\|f\|_{AP}=\sup\limits_{t\in \mathbb{R}}\|f(t)\|_{\mathbb{X}}$ for $f\in AP(\mathbb{R},\mathbb{X})$. And if $f\in AP(\mathbb{R},\mathbb{X})$, then the function $f$ is bounded and uniformly continuous on $\mathbb{R}$ in the norm $\|\cdot\|_{AP}$.
\end{lemma}

\begin{theorem}(\cite{4}) (Bochner property) Let $f\in BC(\mathbb{R},\mathbb{X})$, then $f$ belongs to $AP(\mathbb{R},\mathbb{X})$ if only if the family $\mathcal{F}=\{f(t+h);h\in \mathbb{R}\}$ is relatively compact in $BC(\mathbb{R},\mathbb{X})$.
\end{theorem}

\begin{lemma}(\cite{4}) The space $M_{\flat}(\mathbb{R},\mathbb{X})$ is a complete space with the seminorm $\|\cdot\|_{M_{\flat}}$.
\end{lemma}

\section{Besicovitch almost periodic functions and their basic properties}
\setcounter{equation}{0}

In this section, we start with the Bohr definition of  Besicovitch almost periodic functions   and derive some important properties of  Besicovitch almost periodic functions in Bohr's sense, including translation invariance and composition theorem, and obtain the equivalence between Bohr's definition and Bochner's definition of almost periodic functions.

Let
$$L_{\flat}:=\bigg\{f:f\in M_{\flat}(\mathbb{R},\mathbb{X}),~\limsup\limits_{T\rightarrow +\infty}\frac{1}{2T}\int_{-T}^{T}\|f(t)\|_{\mathbb{X}}^{\flat}dt=0\bigg\}.$$

\begin{lemma}\cite{4} The set $L_{\flat}$ is a closed linear manifold in $M_{\flat}(\mathbb{R},\mathbb{X})$.
\end{lemma}
We shall introduce a relation in $M_{\flat}(\mathbb{R},\mathbb{X})$ as follows:
$$u\simeq v,~u,v\in M_{\flat}(\mathbb{R},\mathbb{X}),~if~u-v\in L_{\flat}.$$\par
It is easy to check that this relation, symbolized by $\simeq$, is indeed an equivalence relation.

The quotient space $M_{\flat}(\mathbb{R},\mathbb{X})/L_{\flat}$ is the set of equivalence classes with respect to the relation $\simeq$, organized in accordance with the operations $[u]+[v]=[u+v]$, $u,v\in M_{\flat}(\mathbb{R},\mathbb{X})$, and $\lambda[u]=[\lambda u]$ for $u\in M_{\flat}(\mathbb{R},\mathbb{X})$ and for any scalar $\lambda\in \mathbb{R}$.

\begin{definition}\cite{4} The quotient space $M_{\flat}(\mathbb{R},\mathbb{X})/L_{\flat}$ is called Marcinkiewicz space and denoted by $\mathcal{M}_{\flat}$.\end{definition}

\begin{lemma}\cite{4} The Marcinkiewicz function space $\mathcal{M}_{\flat}(\mathbb{R},\mathbb{X})$ is a Banach space with the norm defined by $$\|[x]\|_{\mathcal{M}_{\flat}}=\|x+L_{\flat}\|_{\mathcal{M}_{\flat}}:=\inf\{\|y\|_{{M}_{\flat}}:y\in x+L_{\flat}\},$$
where $[x]\in \mathcal{M}_{\flat}.$
\end{lemma}
\begin{remark}It is easy to check $\|x+L_{\flat}\|_{\mathcal{M}_{\flat}}=\|x\|_{{M}_{\flat}}$ for any $x\in {M}_{\flat}(\mathbb{R},\mathbb{X})$.
\end{remark}

\begin{remark}For $f,g\in {\mathcal{M}_{\flat}}(\mathbb{R},\mathbb{X})$, we have $(f+L_{\flat})-(g+L_{\flat})=f-g+L_{\flat}$.
\end{remark}

\begin{definition}\label{defb} (Bohr's definition) Let $f\in \mathcal{M}_{\flat}(\mathbb{R},\mathbb{X})$, then $f$ is called Besicovitch almost periodic   if the following properties hold true:
\begin{itemize}
\item [$(1)$]$f$ is uniformly continuous in the $\mathcal{M}_{\flat}$-norm, that is, for any $\varepsilon>0$ there exists $\delta=\delta(\varepsilon)>0$ such that
 $$\|f(t+h)-f(t)\|_{\mathcal{M}_{\flat}}<\varepsilon,$$
 for all $h\in \mathbb{R}$, $|h|<\delta$,
\item [$(2)$] for any $\varepsilon>0$, it is possible to find a real number $l=l(\varepsilon)>0$, for any interval with length $l$, there exists a number $\tau=\tau(\varepsilon)$ in this interval such that $$\|f(t+\tau)-f(t)\|_{\mathcal{M}_{\flat}}<\varepsilon.$$
\end{itemize}
The number $\tau$ is then called a $\varepsilon$-translation number of $f$. The collection of all Besicovitch almost periodic functions will be denoted by $B_{AP}^{\flat}(\mathbb{R},\mathbb{X})$. For $f\in B_{AP}^{\flat}(\mathbb{R},\mathbb{X})$, we define   $\|f\|_{B^{\flat}}=\|f\|_{\mathcal{M}_{\flat}}$.
\end{definition}

We denote by $UC\mathcal{M}_{\flat}(\mathbb{R},\mathbb{X})$ the collection of functions $f\in \mathcal{M}_{\flat}$ that satisfy condition (1) in Definition \ref{defb}.

\begin{remark}It is worth noting that, in general, when we write that a function belongs to $B_{AP}^{\flat}(\mathbb{R},\mathbb{X})$ we do not have in mind the function $f$ itself, it does represent a whole class that is equivalent to the function $f$.
\end{remark}

\begin{lemma}\label{lem33} If $f\in B_{AP}^{\flat}(\mathbb{R},\mathbb{X})$ and $\alpha \in \mathbb{R}$, then we have $f(\cdot - \alpha)\in B_{AP}^{\flat}(\mathbb{R},\mathbb{X})$.
\end{lemma}
\begin{proof}Using the fact that $f\in B_{AP}^{\flat}(\mathbb{R},\mathbb{X})$, we have for any $\varepsilon>0$, it is possible to find a real number $l=l(\varepsilon)>0$, for any interval with length $l$, there exists a number $\tau=\tau(\varepsilon)$ in this interval such that $$\|f(t+\tau)-f(t)\|_{B^{\flat}}<\varepsilon.$$
Since
\begin{align*}
\|f(t+\tau+\alpha)-f(t+\alpha)\|_{B^{\flat}}&=\left(\limsup\limits_{T\rightarrow+\infty}\frac{1}{2T}\int_{-T}^{T}\|f(t+\tau+\alpha)-f(t+\alpha)\|_{\mathbb{X}}^{\flat}dt\right)^{\frac{1}{\flat}}\\
&=\left(\limsup\limits_{T\rightarrow +\infty}\frac{1}{2T}\int_{-T+\alpha}^{T+\alpha}\|f(t+\tau)-f(t)\|_{\mathbb{X}}^{\flat}dt\right)^{\frac{1}{\flat}}\\
&\leq\left(\limsup\limits_{T\rightarrow +\infty}\frac{1}{2T}\int_{-T-\alpha}^{T+\alpha}\|f(t+\tau)-f(t)\|_{\mathbb{X}}^{\flat}dt\right)^{\frac{1}{\flat}}<\varepsilon.
\end{align*}
Hence,  $f(\cdot - \alpha)\in B_{AP}^{\flat}(\mathbb{R},\mathbb{X})$. The proof is complete.
\end{proof}
\begin{remark}
In the proof of Lemma \ref{lem33}, the $f$'s in the left side of the formula are  equivalence classes and the $f$'s in the right side are representative elements of the equivalence classes.   We will not make an explanation later, it is clear from the context.
\end{remark}
\begin{lemma}\label{lem34} If $f\in B_{AP}^{\flat}(\mathbb{R},\mathbb{X})$ and $\lambda \in \mathbb{R}$, then we have $\lambda f\in B_{AP}^{\flat}(\mathbb{R},\mathbb{X})$.
\end{lemma}
\begin{proof}Since $f\in B_{AP}^{\flat}(\mathbb{R},\mathbb{X})$, for $\varepsilon>0$, there exists $l=l(\varepsilon)>0$ such that every interval of length $l$ contains a $\tau$ with the property that
$$\|f(t+\tau)-f(t)\|_{B^{\flat}}<\varepsilon.$$
Now
$$\|\lambda f(t+\tau)-\lambda f(t)\|_{B^{\flat}}=|\lambda|\cdot\| f(t+\tau)- f(t)\|_{B^{\flat}}<|\lambda|\varepsilon.$$
Hence, $\lambda f\in B_{AP}^{\flat}(\mathbb{R},\mathbb{X})$. The proof is complete.
\end{proof}

\begin{lemma}\label{lem35}Let $f\in B_{AP}^{\flat}(\mathbb{R},\mathbb{X})$, then the family $\mathcal{F}=\{f(t+h);h\in \mathbb{R}\}$ is relatively compact in the topology of $\mathcal{M}_{\flat}$.
\end{lemma}
\begin{proof} Let $\{h_{n}\}_{n\in \mathbb{N}}\subset \mathbb{R}$, then $\{f(t+h_{n});~n\geq1\}\subset\mathcal{F}$ is a sequence of translates of $f$. Since $f\in B_{AP}^{\flat}(\mathbb{R},\mathbb{X})$, one can find a real number $l=l(\varepsilon)>0$, for any interval with length $l$, there exists a number $\tau_n=\tau_n(\varepsilon)$ in this interval $[-h_{n},-h_{n}+l]$ such that $$\|f(t+\tau_n)-f(t)\|_{B^{\flat}}<\varepsilon.$$
By Definition \ref{defb}, we see that $f$ is uniformly continuous in the norm $\|\cdot\|_{B^{\flat}}$. Thus, there exists $\delta=\delta(\varepsilon)>0$ such that
 $$\|f(t+h)-f(t)\|_{B^{\flat}}<\varepsilon,$$
 for all $h\in \mathbb{R}$, $|h|<\delta$.
And for any $h_{n}\in \mathbb{R}$, we have $\tau_n+h_{n}\in [0,l]$, this guarantees the existence of a subsequence $\{h_{n_{k}};k\geq1\}\subset \{h_{n}\}_{n\in \mathbb{N}}$ such that $\{\tau_n+h_{n_{k}};k\geq1\}$ is convergent. Hence, there exists a positive integer $N=N(\varepsilon)$ such that
$$|\tau_n+h_{n_{k}}-\tau_n+h_{n_{k'}}|<\delta,$$
for $k,~k'>N$. Then we obtain
\begin{equation*}
\begin{split}
\|f(t+h_{n_{k}})-f(t+h_{n_{k'}})\|_{B^{\flat}}
&\leq\|f(t+h_{n_{k}})-f(t+\tau_n+h_{n_{k}})\|_{B^{\flat}}\\
&~~+\|f(t+\tau_n+h_{n_{k}})-f(t+\tau_n+h_{n_{k'}})\|_{B^{\flat}}\\
&~~+\|f(t+\tau_n+h_{n_{k'}})-f(t+h_{n_{k'}})\|_{B^{\flat}}\\
&<\varepsilon+\varepsilon+\varepsilon=3\varepsilon,
\end{split}
\end{equation*}
which yields $\{f(t+h_{n_{k}});k\geq1\}$ is uniformly convergent in the $B^{\flat}$-norm, therefore, the family $\mathcal{F}=\{f(t+h);h\in \mathbb{R}\}$ is relatively compact. The proof is complete.
\end{proof}

\begin{definition}(Bochner definition) Function $f\in UC\mathcal{M}_{\flat}(\mathbb{R},\mathbb{X})$ is called Besicovitch almost periodic if for any sequence $\{h'_{n}\}_{n\in \mathbb{N}}$ of real numbers, there exists a subsequence $\{h_{n}\}_{n\in \mathbb{N}}$ of $\{h'_{n}\}_{n\in \mathbb{N}}$ such that $\{f(t+h_{n});n\in \mathbb{N}\}$ converges in the norm $\|\cdot\|_{\mathcal{M}_{\flat}}$.
\end{definition}

\begin{lemma}\label{lem36}The Bohr definition and the Bochner definition of Besicovitch almost periodicity are equivalent.
\end{lemma}
\begin{proof}Let $f\in \mathcal{M}_{\flat}(\mathbb{R},\mathbb{X})$, as we have seen in Lemma 3.5, the Bohr  property of $f$ implies the relative compactness of $\mathcal{F}$, namely, the Bohr definition implies the Bochner definition. To complete the proof, we just need to show that the Bochner definition implies the Bohr definition, that is, we need to prove that
if the set $\mathcal{F}=\{f(t+h);h\in \mathbb{R}\}$  is relatively compact set in $\mathcal{M}_{\flat}(\mathbb{R},\mathbb{X})$, then for any $\varepsilon>0$, there exists $l=l(\varepsilon)>0$ such that in each interval $(a,a+l)\subset \mathbb{R}$, there is a number $\tau$ with the property
\begin{align}\label{3.2}
\|f(t+\tau)-f(t)\|_{{\mathcal{M}}_{\flat}}=\left(\limsup\limits_{T\rightarrow +\infty}\frac{1}{2T}\int_{-T}^{T}\|f(t+\tau)-f(t)\|_{\mathbb{X}}^{\flat}dt\right)^{\frac{1}{\flat}}<\varepsilon.
\end{align}
Let us assume, on the contrary, that $f$ does not possess Bohr's property. Then we can find at least one $\varepsilon_{0}>0$ for which $l=l(\varepsilon_{0})$ does not exist. In other words, for each $l>0$, one can find an interval of length $l$ that contains no points $\tau$ with Bohr's property. Now choose an arbitrary number $h_{1}\in \mathbb{R}$ and an interval $(a_{1},b_{1})\subset \mathbb{R}$ of length larger than $2|h_{1}|$ such that $(a_{1},b_{1})$ does not contain any number $\tau$ with Bohr's property. Denote $h_{2}=\frac{a_{1}+b_{1}}{2}$, then $h_{2}-h_{1}\in (a_{1},b_{1})$, and therefore, $h_{2}-h_{1}$ cannot be taken as $\tau$, which means
$$\|f(t+h_{2}-h_{1})-f(t)\|_{\mathcal{M}_{\flat}}=\left(\limsup\limits_{T\rightarrow +\infty}\frac{1}{2T}\int_{-T}^{T}\|f(t+h_{2}-h_{1})-f(t)\|_{\mathbb{X}}^{\flat}dt\right)^{\frac{1}{\flat}}\geq\varepsilon_{0}.$$
Let us take an interval $(a_{2},b_{2})\subset \mathbb{R}$ of length larger than $2(|h_{1}|+|h_{2}|)$ that does not contain any number $\tau$ with Bohr's property. We continue the process, letting $h_{3}=\frac{a_{2}+b_{2}}{2}$, then we have $h_{3}-h_{2}$, $h_{3}-h_{1}\in (a_{2},b_{2})$, and this means that $h_{3}-h_{2}$ and $h_{3}-h_{1}$ cannot be taken as number $\tau$, that is,
$$\|f(t+h_{3}-h_{1})-f(t)\|_{\mathcal{M}_{\flat}}=\left(\limsup\limits_{T\rightarrow +\infty}\frac{1}{2T}\int_{-T}^{T}\|f(t+h_{3}-h_{1})-f(t)\|_{\mathbb{X}}^{\flat}dt\right)^{\frac{1}{\flat}}\geq\varepsilon_{0}$$
and
$$\|f(t+h_{3}-h_{2})-f(t)\|_{\mathcal{M}_{\flat}}=\left(\limsup\limits_{T\rightarrow +\infty}\frac{1}{2T}\int_{-T}^{T}\|f(t+h_{3}-h_{2})-f(t)\|_{\mathbb{X}}^{\flat}dt\right)^{\frac{1}{\flat}}\geq\varepsilon_{0}.$$
Proceeding similarly, we construct the numbers $h_{4},~h_{5},\ldots,$ with the property that none of the differences $h_{i}-h_{j}$, $i>j$, could be taken as number $\tau$ in formula \eqref{3.2}. Hence, for $i>j$, we have
\begin{align*}
\|f(t+h_{i}-h_{j})-f(t)\|_{\mathcal{M}_{\flat}}&=\left(\limsup\limits_{T\rightarrow +\infty}\frac{1}{2T}\int_{-T}^{T}\|f(t+h_{i}-h_{j})-f(t)\|_{\mathbb{X}}^{\flat}dt\right)^{\frac{1}{\flat}}\\
&=\left(\limsup\limits_{T\rightarrow +\infty}\frac{1}{2T}\int_{-T-h_{j}}^{T-h_{j}}\|f(t+h_{i})-f(t+h_{j})\|_{\mathbb{X}}^{\flat}dt\right)^{\frac{1}{\flat}}\geq\varepsilon_{0}.
\end{align*}
Consequently, we have
\begin{align*}
\|f(t+h_{i})-f(t+h_{j})\|_{\mathcal{M}_{\flat}}&=\left(\limsup\limits_{T\rightarrow +\infty}\frac{1}{2(T+h_{j})}\int_{-T-h_{j}}^{T+h_{j}}\|f(t+h_{i})-f(t+h_{j})\|_{\mathbb{X}}^{\flat}dt\right)^{\frac{1}{\flat}}\\
&=\left(\limsup\limits_{T\rightarrow +\infty}\frac{1}{2T}\int_{-T-h_{j}}^{T+h_{j}}\|f(t+h_{i})-f(t+h_{j})\|_{\mathbb{X}}^{\flat}dt\right)^{\frac{1}{\flat}}\\
&\geq\left(\limsup\limits_{T\rightarrow +\infty}\frac{1}{2T}\int_{-T-h_{j}}^{T-h_{j}}\|f(t+h_{i})-f(t+h_{j})\|_{\mathbb{X}}^{\flat}dt\right)^{\frac{1}{\flat}}\geq\varepsilon_{0},
\end{align*}
which contradicts the property of relative compactness of the family $\mathcal{F}=\{f(t+h);h\in \mathbb{R}\}$. Hence, $f$ possess  Bohr's property. The proof is complete.
\end{proof}

\begin{lemma}\label{lem37} If $\{f_{n}\}_{n\in \mathbb{N}}\subset B_{AP}^{\flat}(\mathbb{R},\mathbb{X})$ such that there exists $f\in \mathcal{M}_{\flat}(\mathbb{R},\mathbb{X})$ with $\|f_{n}-f\|_{B^{\flat}}\rightarrow 0$ as $n\rightarrow\infty$, then $f\in B_{AP}^{\flat}(\mathbb{R},\mathbb{X})$.
\end{lemma}
\begin{proof}Obviously, $f$ is uniformly continuous in the norm $\|\cdot\|_{B^{\flat}}$. Since $\{f_{n}\}_{n\in \mathbb{N}}\subset B_{AP}^{\flat}(\mathbb{R},\mathbb{X})$ such that $\|f_{n}-f\|_{B^{\flat}}\rightarrow 0$ as $n\rightarrow\infty$, then for any $\varepsilon>0$, there corresponds a large enough positive integer $N_{1}=N_{1}(\varepsilon)$ such that $$\|f_{N_{1}}-f\|_{B^{\flat}}<\varepsilon.$$ 
Noting that $f_{N_{1}}\in B_{AP}^{\flat}(\mathbb{R},\mathbb{X})$, there exists $l=l(\varepsilon)$ such that every interval of length $l$ contains a number $\tau$ with the property $$\|f_{N_{1}}(t+\tau)-f_{N_{1}}(t)\|_{B^{\flat}}<\varepsilon.$$
And we keep in mind the inequality
\begin{equation*}
 \begin{split}
 &\quad \ \|f (t+\tau)-f(t)\|_{B^{\flat}}\\
 &\leq\|f (t+\tau)-f_{N_{1}}(t+\tau)\|_{B^{\flat}}+ \|f_{N_{1}} (t+\tau)-f_{N_{1}}(t)\|_{B^{\flat}}+\|f_{N_{1}} (t)-f(t)\|_{B^{\flat}}\\
 &<\varepsilon+\varepsilon+\varepsilon=3\varepsilon,
 \end{split}
 \end{equation*}
 which yields $f\in B_{AP}^{\flat}(\mathbb{R},\mathbb{X})$. The proof is complete.
\end{proof}

Similar to the proof of Proposition 3.21 in \cite{4}, one can prove
\begin{lemma}\label{lem39} Let $f_{k}\in B_{AP}^{\flat}(\mathbb{R},\mathbb{X})$, $k=1,2,\ldots,n$. Then, for every $\varepsilon>0$, there exist common $\varepsilon$-translation numbers for these functions.
\end{lemma}

\begin{lemma}\label{lem310} If $f$, $g\in B_{AP}^{\flat}(\mathbb{R},\mathbb{X})$, then we have $f+g\in B_{AP}^{\flat}(\mathbb{R},\mathbb{X})$.
\end{lemma}
\begin{proof}According to Lemma \ref{lem39}, for any $\varepsilon>0$,  there exists $l=l(\varepsilon)>0$ such that every interval of length $l$ contains a $\tau$ with the property
$$\|f(t+\tau)-f(t)\|_{B^{\flat}}<\varepsilon,~\|g(t+\tau)-g(t)\|_{B^{\flat}}<\varepsilon.$$
Hence, we have
\begin{align*}
&\quad \ \|f(t+\tau)g(t+\tau)-f(t)g(t)\|_{B^{\flat}}\\
&\leq\|f(t+\tau)g(t+\tau)-f(t)g(t+\tau)\|_{B^{\flat}}+\|f(t)g(t+\tau)-f(t)g(t)\|_{B^{\flat}}\\
&\leq\|f(t+\tau)-f(t)\|_{B^{\flat}}\|g(t+\tau)\|_{B^{\flat}}+\|f(t)\|_{B^{\flat}}\|g(t+\tau)-g(t)\|_{B^{\flat}}\\
&\leq\varepsilon\|g(t+\tau)\|_{B^{\flat}}+\varepsilon\|f(t)\|_{B^{\flat}},
\end{align*}
therefore, $f+g\in B_{AP}^{\flat}(\mathbb{R},\mathbb{X})$. The proof is complete.
\end{proof}

\begin{lemma}\label{lem38} The space $B_{AP}^{\flat}(\mathbb{R},\mathbb{X})$ is a Banach space with the norm $\|\cdot\|_{B^{\flat}}$.
\end{lemma}
\begin{proof} Clearly, $B_{AP}^{\flat}(\mathbb{R},\mathbb{X})\subset \mathcal{M}_{\flat}(\mathbb{R},\mathbb{X})$ and by Lemma 3.7, the space $B_{AP}^{\flat}(\mathbb{R},\mathbb{X})$ is closed. Therefore, $B_{AP}^{\flat}(\mathbb{R},\mathbb{X})$ is also a Banach space. The proof is complete.
\end{proof}

\begin{remark} Unlike in the case of Bohr almost periodic functions, the product $f\cdot g$ does not necessarily belong to $B_{AP}^{\flat}(\mathbb{R},\mathbb{X})$. But we replace a Besicovitch almost periodic function with a Bohr almost periodic function, the result is still valid.
\end{remark}
\begin{lemma}\label{lem311} If $f\in AP(\mathbb{R},\mathbb{X})$ and $g\in B_{AP}^{\flat}(\mathbb{R},\mathbb{X})$, then we have $f \cdot g \in B_{AP}^{\flat}(\mathbb{R},\mathbb{X})$.
\end{lemma}
\begin{proof}Clearly, we have $f\cdot g\in \mathcal{M}_{\flat}$ and $f\cdot g$ is uniformly continuous in the norm $\|\cdot\|_{B^{\flat}}$. Since $f\in AP(\mathbb{R},\mathbb{X})$, for any sequence $\{h_{n}\}_{n\in \mathbb{N}}$ of real numbers there exists a subsequence $\{h_{1n}\}_{n\in \mathbb{N}}$ of $\{h_{n}\}_{n\in \mathbb{N}}$ such that $\{f(t+h_{1n});n\in \mathbb{N}\}$ converges uniformly on $\mathbb{R}$. Since $g\in B^{\flat}(\mathbb{R},\mathbb{X})$ it follows that for the sequence $\{g(t+h_{1n});n\in \mathbb{N}\}$, there exists a subsequence $\{h_{2n}\}_{n\in \mathbb{N}}\subset\{h_{1n}\}_{n\in \mathbb{N}}$ such that the sequence of functions $\{g(t+h_{2n});n\in \mathbb{N}\}$ converges in the norm $\|\cdot\|_{B^{\flat}}$. It is then clear that $\{f(t+h_{2n});n\in \mathbb{N}\}$ converges uniformly in $t\in \mathbb{R}$.\par
Now for any $\varepsilon>0$, there exists a large enough integer $N=N(\varepsilon)$ such that
 \begin{equation*}
 \begin{split}
 &\quad \ \|f (t+h_{2n})g(t+h_{2n})-f(t)g(t)\|_{B^{\flat}}\\
 &\leq\|f (t+h_{2n})g(t+h_{2n})-f(t+h_{2n})g(t)\|_{B^{\flat}}+\|f (t+h_{2n})g(t)-f(t)g(t)\|_{B^{\flat}}\\
 &\leq \|f(t+h_{2n})\|_{B^{\flat}}\|g(t+h_{2n})- g(t)\|_{B^{\flat}}+\|f(t+h_{2n})- f(t)\|_{B^{\flat}}\|g(t)\|_{B^{\flat}}\\
 &\leq\|f(t+h_{2n})\|_{AP}\|g(t+h_{2n})- g(t)\|_{B^{\flat}}+\|f(t+h_{2n})- f(t)\|_{AP}\|g(t)\|_{B^{\flat}}\\
&<\|f\|_{AP}\varepsilon+\varepsilon\|g(t)\|_{B^{\flat}},
 \end{split}
 \end{equation*}
for $n\geq N$, which means that $\{f(t+h_{2n})g(t+h_{2n});n\in \mathbb{N}\}$ converges in the norm $\|\cdot\|_{B^{\flat}}$. Hence, $f\cdot g$ satisfies the Bochner definition, namely, $f\cdot g \in B_{AP}^{\flat}(\mathbb{R},\mathbb{X})$. The proof is complete.
 \end{proof}

 \begin{lemma}\label{lem312}If $x\in AP(\mathbb{R},\mathbb{R})$ and $f\in B_{AP}^{\flat}(\mathbb{R},\mathbb{X})$, then $f(\cdot-x(\cdot))\in B_{AP}^{\flat}(\mathbb{R},\mathbb{X})$.
\end{lemma}
\begin{proof}
Step 1. We will prove $f(\cdot-x(\cdot))\in \mathcal{M}_{\flat}(\mathbb{R},\mathbb{X})$ is uniformly continuous in the norm $\|\cdot\|_{B^{\flat}}$.\par
Since $f\in B_{AP}^{\flat}(\mathbb{R},\mathbb{X})$ and $x\in AP(\mathbb{R},\mathbb{R})$, we know that $f$ and $x$ are uniformly continuous in their corresponding norms. Then for any $\varepsilon>0$, there exists $\delta=\delta(\varepsilon)>0$ such that
$$\|f(t+h)-f(t)\|_{B^{\flat}}<\varepsilon$$
and
$$\|x(t+h)-x(t)\|_{AP}<\varepsilon,$$
for $|h|<\delta, h\in \mathbb{R}$.
Since $x\in AP(\mathbb{R},\mathbb{R})$, the set $\{x(t):t\in \mathbb{R}\}\subset \mathbb{R}$ is relatively compact, the closure of $\{x(t):t\in \mathbb{R}\}$, denoted by $\overline{x(\mathbb{R})}$, is a compact set in $\mathbb{R}$. So we can find a finite number of points $t_{k}\in \mathbb{R}$, $k=1,2,\ldots,m$, such that the union of these finite number's  open balls $B(x({t_{k}}),\delta)$ covers the set $\overline{x(\mathbb{R})}$. For each $t\in \mathbb{R}$, there exists $k'$, $1\leq k'\leq m$, such that $x(t)\in B(x({t_{k'}}),\delta)$. Then we obtain
$$\|f(t-x(t))\|_{B^{\flat}}\leq\|f(t-x(t))-f(t-x(t_{k'}))\|_{B^{\flat}}+\|f(t-x(t_{k'}))\|_{B^{\flat}}<+\infty,$$
from this, $f(\cdot-x(\cdot))$ is bounded. Therefore, $f(\cdot-x(\cdot))\in \mathcal{M}_{\flat}(\mathbb{R},\mathbb{X})$.
In view of the continuity of $f$ and $x$, taking $\varepsilon=\delta$, we get
\begin{align*}
&\quad \ \|f(t+h-x(t+h))-f(t-x(t))\|_{B^{\flat}}\\
&\leq\|f(t+h-x(t+h))-f(t+h-x(t))\|_{B^{\flat}}\\
&\quad+\|f(t+h-x(t)-f(t-x(t)\|_{B^{\flat}}\\
&<\varepsilon+\varepsilon=2\varepsilon.
\end{align*}
Hence $f(\cdot-x(\cdot))$ is uniformly continuous.

Step 2. We will prove $f(\cdot-x(\cdot))\in B_{AP}^{\flat}(\mathbb{R},\mathbb{X})$.

To obtain the result, we only need to show  the family $\mathcal{F}=\{f(t+h-x(t+h)):h\in \mathbb{R}\}$ is relatively compact.

Since $f\in B_{AP}^{\flat}(\mathbb{R},\mathbb{X})$, $x\in AP(\mathbb{R},\mathbb{R})$, for any $\varepsilon>0$ and $\{h_{q}\}_{q\in \mathbb{N}}\subset \mathbb{R}$, we can write
$$\|f(t+h_{1q})-f(t+h_{1q'})\|_{B^{\flat}}<\varepsilon, \quad \mathrm{for} \quad q,q'\geq N_{1}(\varepsilon),$$
$$|x(t+h_{1q})-x(t+h_{1q'})|<\delta(\varepsilon),\quad \mathrm{for}\quad q,q'\geq N_{2}(\varepsilon),$$
where $\{h_{1q}\}_{q\in \mathbb{N}},\{h_{1q'}\}_{q\in \mathbb{N}}\subset\{h_{q}\}_{q\in \mathbb{N}}$.
By the uniform  continuity of $f$ in the norm $\|\cdot\|_{B^{\flat}}$, denote $N=\max\{N_{1},N_{2}\}$, we obtain
\begin{align*}
&\quad \ \|f(t+h_{1q}-x(t+h_{1q}))-f(t+h_{1q'}-x(t+h_{1q'}))\|_{B^{\flat}}\\
&\leq\|f(t+h_{1q}-x(t+h_{1q}))-f(t+h_{1q'}-x(t+h_{1q}))\|_{B^{\flat}}\\
&\quad+\|f(t+h_{1q'}-x(t+h_{1q})-f(t+h_{1q'}-x(t+h_{1q'})\|_{B^{\flat}}\\
&<\varepsilon+\varepsilon=2\varepsilon,
\end{align*}
for~$q,~q'\geq N$, which means that the function $f(\cdot-x(\cdot))$ meets the Bochner definition, that is, $f(\cdot-x(\cdot))\in B_{AP}^{\flat}(\mathbb{R},\mathbb{X})$. In summary, the proof is complete.
\end{proof}
Let $(\mathbb{Y},\|\cdot\|_\mathbb{Y})$ be a Banach space.
\begin{definition} A function $f:\mathbb{R}\times \mathbb{X}\rightarrow \mathbb{Y},(t,x)\mapsto f(t,x)$ with $f(\cdot,x)\in \mathcal{M}_{\flat}$ for each $x\in \mathbb{X}$, is said to be Besicovitch almost periodic in $t\in \mathbb{R}$ uniformly in $x\in \mathbb{X}$ if the following properties hold true:
\begin{itemize}
\item [$(1)$] For every bounded subset $\mathbb{B}\subset \mathbb{X}$, $f(t,x)$ is uniformly continuous in $x\in \mathbb{B}$ with respect to the norm $\|\cdot\|_{B^{\flat}}$  uniformly for $t\in \mathbb{R}$.
\item [$(2)$]For each $\varepsilon>0$ and each bounded subset $\mathbb{B}\subset \mathbb{X}$, there exists $l=l(\varepsilon)>0$ such that every interval of length $l$ contains a number $\tau$ with the property
\begin{equation*}
\left(\limsup\limits_{T\rightarrow +\infty}\frac{1}{2T}\int_{-T}^{T}\|f(t+\tau,x)-f(t,x)\|_{\mathbb{Y}}^{\flat}dt\right)^{\frac{1}{\flat}}<\varepsilon,
\end{equation*}
uniformly in $x\in \mathbb{B}$.
\end{itemize}
\end{definition}
The collection of these functions will be  denoted by $B_{AP}^{\flat}(\mathbb{R}\times \mathbb{X},\mathbb{Y})$.

\section{ Besicovitch almost periodic solutions}
\setcounter{equation}{0}

In this section, we study the existence and uniqueness of   Besicovitch almost periodic mild solutions to equation
\begin{equation}\label{eq1}
  x'(t)=Ax(t)+F(t,x(t),x(t-\tau(t))),
\end{equation}
where $A$ is the infinitesimal generator of a $C_0$-semigroup $\{T(t) : t\geq 0\}$ on a Banach space $\mathbb{X}$, $F: \mathbb{R}\times \mathbb{X}\times \mathbb{X}\rightarrow \mathbb{X}$ is a measurable function.

We call $x:\mathbb{R}\rightarrow \mathbb{X}$ a mild solution of \eqref{eq1}, if for $x(\theta)=\varphi(\theta), \theta\in [t_0-\bar{\tau},t_0]$, where $\varphi\in C(t_0-\bar{\tau},t_0]$ and $\bar{\tau}=\sup\limits_{t\in \mathbb{R}}\tau(t)$,   it satisfies
\begin{equation*}
  x(t)=T(t-t_0)x(t_0)+\int_{t_0}^{t}T(t-s)F(s,x(s),x(s-\tau(s)))ds,\,\, t\geq t_0.
\end{equation*}

\begin{remark}
Although the space  $(B_{AP}^{\flat}(\mathbb{R},\mathbb{X}), \|\cdot\|_{B^\flat}$ is a Banach space, we can't directly apply the fixed point theorem on a subset of $B_{AP}^{\flat}(\mathbb{R},\mathbb{X})$ to study the existence of Besicovitch almost periodic solutions of \eqref{eq1}. Because even if we have derived that an operator $\Phi$ defined on a subset of $B_{AP}^{\flat}(\mathbb{R},\mathbb{X})$  has a fixed point $x^*$ such that
\begin{equation}\label{req1}
  \Phi x^* = x^*,
\end{equation}
we could not conclude that there exists a concrete function $y\in  Fx^*$ such that $\Phi y = y$ because \eqref{req1} only means that the equivalence class with $\Phi x^*$ as the representative element is the same as the equivalence class with $x^*$ as the representative element.
\end{remark}

In what follows, in order to distinguish, we denote
\begin{align*}
L_{\flat_1}:=\bigg\{f:f\in M_{\flat}(\mathbb{R},\mathbb{X}),~\limsup\limits_{T\rightarrow +\infty}\frac{1}{2T}\int_{-T}^{T}\|f(t)\|_{\mathbb{X}}^{\flat}dt=0\bigg\},
\end{align*}
and $L_{\flat_2}$ is the collection of all functions $f\in M_{\flat}(\mathbb{R}\times\mathbb{X}^{2},\mathbb{X})$ satisfying $$\limsup\limits_{T\rightarrow +\infty}\frac{1}{2T}\int_{-T}^{T}\|f(t,x)\|_{\mathbb{X}}^{\flat}dt=0$$
uniformly in $x\in \mathbb{B}$, where $\mathbb{B}$ is an arbitrary bounded subset of $\mathbb{X}^2$.
The quotient spaces corresponding to $L_{\flat_1}$ and $L_{\flat_2}$ are $\mathcal{M}_{\flat_1}$ and $\mathcal{M}_{\flat_2}$, respectively.

 We make some assumptions:
\begin{itemize}
 \item [$(H_1)$] Functions  $\tau\in AP(\mathbb{R},\mathbb{R}^{+})$ and  $F$ satisfies $F+L_{\flat_2}\in B_{AP}^{\flat}(\mathbb{R}\times \mathbb{X}^2,\mathbb{X})$.
 \item [$(H_2)$] The function  $F(t,x,y)$ is Lipschitz in $x,y\in \mathbb{X}$ uniformly in $t\in \mathbb{R}$, that is, there exist positive constants $\mathfrak{L}_{1}$ and  $\mathfrak{L}_{2}$ such that for all $x_1,y_1, x_2,y_2\in \mathbb{X}$ and $t\in \mathbb{R}$,
     $$\|F(t,x_1,y_1)-F(t,x_2,y_2)\|_\mathbb{X}\leq \mathfrak{L}_{1}\|x_1-x_2\|_\mathbb{X}+\mathfrak{L}_{2}\|y_1-y_2\|_\mathbb{X}$$
     and $F(t,0,0)=0$.

  \item [$(H_3)$]    $A$ is the infinitesimal generator of an exponentially stable $C_0$-semigroup
$\{T(t) : t \geq 0\}$, that is, there exist numbers $N,\lambda>0$ such that $\|T(t)\|\leq Ne^{-\lambda t}, t\geq 0$.

\item [$(H_4)$]The constant $\kappa:= \frac{N(\mathfrak{L}_1+\mathfrak{L}_2)}{\lambda}<1$, where $N$ is mentioned in $(H_3)$.
\end{itemize}

Let $$\mathbb{W}=\{y\in L^{\infty}(\mathbb{R},\mathbb{X}):\,y+L_{\flat_1}\in B_{AP}^{\flat}(\mathbb{R},\mathbb{X})\}$$
 with the norm $\|\cdot\|_{\mathbb{W}}:=\|\cdot\|_{\infty}$.

\begin{lemma}\label{lem41} $(\mathbb{W},\|\cdot\|_{\mathbb{W}})$ is a Banach space.
\end{lemma}
\begin{proof} Let $\{f_{n};n\geq1\}$ be  an arbitrary Cauchy sequence in $\mathbb{W}$. Then, for each $\varepsilon>0$, one can find an integer $N=N(\varepsilon)$ such that $$\|f_{n}-f_{m}\|_{\mathbb{W}}<\varepsilon, \,\,n,m>N.$$
Since $(L^{\infty}(\mathbb{R},\mathbb{X}),\|\cdot\|_{\infty})$ is a Banach space and $\mathbb{W} \subset L^{\infty}(\mathbb{R},\mathbb{X})$ it follows that there exists $f\in L^{\infty}(\mathbb{R},\mathbb{X})$ such that $\|f_{n}-f\|_{\infty}\rightarrow0$ as $n\rightarrow\infty$. Noting that $\|f_{n}-f\|_{M_{\flat}}\leq\|f_{n}-f\|_{\infty}$, therefore,
\begin{equation}\label{4.5}
 \|f_{n}-f\|_{M_{\flat}}\rightarrow0 \quad \mathrm{as} \,\, n\rightarrow\infty,
\end{equation}
 which means that $f\in M_{\flat}(\mathbb{R},\mathbb{X})$ from the completeness of $ M_{\flat}(\mathbb{R},\mathbb{X})$. Further, we have $f+L_{\flat_1}\in \mathcal{M}_{\flat}(\mathbb{R},\mathbb{X})$. From formula \eqref{4.5}, for any $\varepsilon>0$, there exists a large enough $N_{1}=N_{1}(\varepsilon)$ such that $\|f_{N_{1}}-f\|_{M_{\flat}}<\varepsilon$.
Since $f_{N_{1}}+L_{\flat_1}\in B_{AP}^{\flat_1}(\mathbb{R},\mathbb{X})$ is uniformly continuous in the norm $\|\cdot\|_{B^{\flat}}$, there exists $\delta=\delta(\varepsilon)$ such that
$$\|f_{N_{1}}(t)+L_{\flat_1}-(f_{N_{1}}(t+h)+L_{\flat_1})\|_{B^{\flat}}<\varepsilon,$$
for $|h|<\delta$. Now
\begin{align*}
\|f(t+h)+L_{\flat_1}-(f(t)+L_{\flat_1})\|_{\mathcal{M}_{\flat}}&\leq\|f(t+h)+L_{\flat_1}-(f_{N_{1}}(t+h)+L_{\flat_1})\|_{\mathcal{M}_{\flat}}\\
&\quad+\|f_{N_{1}}(t+h)+L_{\flat_1}-(f_{N_{1}}(t)+L_{\flat_1})\|_{\mathcal{M}_{\flat}}\\
&\quad+\|f_{N_{1}}(t)+L_{\flat_1}-(f(t)+L_{\flat_1})\|_{\mathcal{M}_{\flat}}\\
&<\varepsilon+\varepsilon+\varepsilon=3\varepsilon,
\end{align*}
which implies $f+L_{\flat_1}$ is uniformly continuous in the norm $\|\cdot\|_{\mathcal{M}_{\flat}}$. Denote by $\tau$   the $\varepsilon$-translation number of $f_{N_{1}}+L_{\flat_1}$, then
\begin{align*}
\|f(t+\tau)+L_{\flat_1}-(f(t)+L_{\flat_1})\|_{\mathcal{M}_{\flat}}\leq&\|f(t+\tau)+L_{\flat_1}-(f_{N_{1}}(t+\tau)+L_{\flat_1})\|_{\mathcal{M}_{\flat}}\\
&+\|f_{N_{1}}(t+\tau)+L_{\flat_1}-(f_{N_{1}}(t)+L_{\flat_1})\|_{\mathcal{M}_{\flat}}\\
&+\|f_{N_{1}}(t)+L_{\flat_1}-(f(t)+L_{\flat_1})\|_{\mathcal{M}_{\flat}}\\
<&\varepsilon+\varepsilon+\varepsilon=3\varepsilon,
\end{align*}
which yields $f+L_{\flat_1}\in B_{AP}^{\flat_1}(\mathbb{R},\mathbb{X})$. Consequently, $f\in \mathbb{W}$. The proof is complete.
\end{proof}

\begin{lemma}\label{lem42}Let $(H_1)$ and $(H_2)$ hold. If  $x\in \mathbb{W}$, then  $F(\cdot,x(\cdot),x(\cdot-\tau(\cdot)))\in L^\infty(\mathbb{R},\mathbb{X})$ and $F(\cdot,x(\cdot),x(\cdot-\tau(\cdot)))+L_{\flat_1}\in B_{AP}^{\flat}(\mathbb{R},\mathbb{X})$.
\end{lemma}
\begin{proof} Firstly, we show that $x(\cdot-\tau(\cdot))\in \mathbb{W}$. Since $\tau\in AP(\mathbb{R},\mathbb{R}^{+})$, so $\{t-\tau(t);~t\in \mathbb{R}\}\subset \mathbb{R}$, then
$$\|x(t-\tau(t))\|_\mathbb{X}\leq\|x\|_{\infty},$$
which yields $x(\cdot-\tau(\cdot))\in L^{\infty}(\mathbb{R},\mathbb{X})$. And according to Lemma \ref{lem312}, we know that $x(\cdot-\tau(\cdot))+L_{\flat_1}\in B_{AP}^{\flat}(\mathbb{R},\mathbb{X})$.
Hence, $x(\cdot-\tau(\cdot))\in \mathbb{W}$.

Next, we will prove that $F(\cdot,x(\cdot),x(\cdot-\tau(\cdot)))\in L^\infty(\mathbb{R},\mathbb{H})$ and $F(\cdot,x(\cdot),x(\cdot-\tau(\cdot)))\in {M}_{\flat}$.   By $(H_2)$,  we have for all $t\in \mathbb{R}$,
\begin{align*}
\|F(t,x(t),x(t-\tau(t)))\|_\mathbb{X}
\leq& \mathfrak{L}_1\|x(t)\|_\mathbb{X} +\mathfrak{L}_2\|x(t-\tau(t))\|_\mathbb{X}.
\end{align*}
 The above inequality implies $F(\cdot,x(\cdot),x(\cdot-\tau(\cdot)))\in L^\infty(\mathbb{R},\mathbb{H})$. Moreover, according to the Minkowski inequality, we have
\begin{align*}
&\ \quad\Big(\frac{1}{2T}\int_{-T}^{T}\|F(t,x(t),x(t-\tau(t)))\|_\mathbb{X}^{\flat}dt\Big)^{\frac{1}{\flat}}\\
&\leq \Big(\frac{1}{2T}\int_{-T}^{T} (\mathfrak{L}_1\|x(t)\|_\mathbb{X}+\mathfrak{L}_2\|x(t-\tau(t))\|_\mathbb{X})^{\flat} dt\Big)^{\frac{1}{\flat}}\\
&\leq \mathfrak{L}_1\Big(\frac{1}{2T}\int_{-T}^{T}\|x(t)\|_\mathbb{X}^{\flat}dt\Big)^{\frac{1}{\flat}}+\mathfrak{L}_2\Big(\frac{1}{2T}\int_{-T}^{T}\|x(t-\tau_q(t))\|_\mathbb{X}^{\flat}dt\Big)^{\frac{1}{\flat}},
\end{align*}
thus
\begin{equation*}\label{4.10}
\begin{split}
\|F(\cdot,x(\cdot),x(\cdot-\tau(\cdot)))\|_{M_\flat}\leq \mathfrak{L}_1 \|x\|_{M_\flat}+ \mathfrak{L}_2\|x\|_{M_\flat},
\end{split}
\end{equation*}
which implies $ F(\cdot,x(\cdot),x(\cdot-\tau(\cdot)))\in {M}_{\flat}$.

Finally, we will show that  $F(\cdot,x(\cdot),x(\cdot-\tau(\cdot)))+L_{\flat_1}\in B_{AP}^{\flat}(\mathbb{R},\mathbb{X})$.   Since $x$ and $x(\cdot-\tau(\cdot))\in L^{\infty}(\mathbb{R},\mathbb{X})$, one can find a bounded subset $\mathbb{B}\subset\mathbb{X}^2$ satisfying
$$(x(t),x(t-\tau(t)))\in \mathbb{B},$$
for all $t\in \mathbb{R}$.   In view of Lemma \ref{lem39}, for $\varepsilon>0$, one can find $\ell=\ell(\varepsilon)>0$ such that every interval of length $\ell$  contains an $h$ with the property that
\begin{equation*}
\left(\limsup\limits_{T\rightarrow +\infty}\frac{1}{2T}\int_{-T}^{T}\|F(t+\tau,u,v)-F(t,u,v)\|_{\mathbb{X}}^{\flat}dt\right)^{\frac{1}{\flat}}<\varepsilon,\,\,u,v\in  \mathbb{X},
\end{equation*}
 \begin{align*}
 &\|x(t+h)+L_{\flat_1}-(x(t)+L_{\flat_1})\|_{B^{\flat}}\nonumber\\
 =&\|x(t+h)-x(t)\|_{M_{\flat}}<\varepsilon
 \end{align*}
 and
 \begin{align*}
 &\|x(t+h-\tau(t+h))+L_{\flat_1}-(x(t-\tau(t))+L_{\flat_1})\|_{B^{\flat}}\nonumber\\
 =&\|x(t+h-\tau(t+h))-x(t-\tau(t))\|_{M_{\flat}}<\varepsilon.
 \end{align*}
From above inequalities, for  $t\in \mathbb{R}$ and $|h|<\delta=\min\{\delta_1,\delta_2\}$, we have
\begin{align*}\label{4.13}
&\|F(t+h,x(t+h),x(t+h-\tau(t+h)))+L_{\flat_1}-(F(t,x(t),x(t-\tau(t)))+L_{\flat_1})\|_{\mathcal{M}_{\flat}}\nonumber\\
=&\|F(t+h,x(t+h),x(t+h-\tau(t+h)))-F(t,x(t),x(t-\tau(t)))\|_{{M}_{\flat}}\nonumber\\
\leq&\bigg(\limsup\limits_{T\rightarrow +\infty}\frac{1}{2T}\int_{-T}^{T}\|F(t+h,x(t+h),x(t+h-\tau(t+h)))\nonumber\\
&-F(t,x(t+h),x(t+h-\tau(t+h)))\|_\mathbb{X}^{\flat}dt\bigg)^{\frac{1}{\flat}}\nonumber\\
&+\bigg(\limsup\limits_{T\rightarrow +\infty}\frac{1}{2T}\int_{-T}^{T}\|F(t,x(t+h),x(t+h-\tau(t+h)))
-F(t,x(t),x(t-\tau(t)))\|_{\mathbb{X}}^{\flat}dt\bigg)^{\frac{1}{\flat}}\\
\leq& \varepsilon(1+\mathcal{L}_1+\mathcal{L}_2),
\end{align*}
which implies that $F(\cdot,x(\cdot),x(\cdot-\tau(\cdot)))+L_{\flat_1} \in B_{AP}^{\flat}(\mathbb{R},\mathbb{X})$. The proof is complete.
\end{proof}

\begin{lemma}\label{lem43} Assume that $(H_1)$-$(H_3)$ hold. Let $x\in \mathbb{W}$,  then  the function $U:\mathbb{R}\rightarrow \mathbb{X}$ defined by
\begin{equation}\label{4.6}
U(t)=\int_{-\infty}^{t}T(t-s)F(s,x(s),x(s-\tau(s)))ds
\end{equation}
belongs to $\mathbb{W}$.
\end{lemma}
\begin{proof} In view of Lemma \ref{lem42}, we know that  $F(\cdot,x(\cdot),x(\cdot-\tau(\cdot)))+L_{\flat_1} \in B_{AP}^{\flat}(\mathbb{R},\mathbb{X})$.
 Our first task is to show that the integral on the right hand of formula \eqref{4.6} exists. Since $(H_3)$,    we have
\begin{align*}
 \|U(t)\|_\mathbb{X}=&\left\|\int_{-\infty}^{t}T(t-s)F(s,x(s),x(s-\tau(s)))ds\right\|_\mathbb{X}\\
\leq&\int_{-\infty}^{t}\|T(t-s)\|_\mathbb{X} \|F(s,x(s),x(s-\tau(s)))\|_\mathbb{X}ds \\
 \leq&\int_{-\infty}^{t}Ne^{-\lambda(t-s)}(\mathcal{L}_1\|x(s)\|_\mathbb{X}+\mathcal{L}_2\|x(s-\tau(s))\|_\mathbb{X})ds \\
\leq&\frac{(\mathfrak{L}_1+\mathcal{L}_2)\|x\|_{\mathbb{W}}}{\lambda},
 \end{align*}
which yields that
$$\|U\|_{\infty}\leq \frac{(\mathfrak{L}_1+\mathcal{L}_2)}{\lambda}\|x\|_{\mathbb{W}},$$
that is to say, \eqref{4.6} is well defined, and as a byproduct, we have obtained that $U\in L^\infty(\mathbb{R},\mathbb{H})$.

Next, we will show that $U+L_{\flat_1}\in B_{AP}^{\flat}(\mathbb{R},\mathbb{X})$. Because $x\in AP(\mathbb{R},\mathbb{X})$, then  by Lemma \ref{lem42},  $F(\cdot,x(\cdot),x(\cdot-\tau(\cdot)))+L_{\flat}\in B_{AP}^{\flat}(\mathbb{R},\mathbb{X})$,  so for any $\varepsilon>0$, we denote by $\sigma=\sigma(\varepsilon)>0$   the   $\varepsilon$-translation number of $F(\cdot,x(\cdot),x(\cdot-\tau(\cdot)))+L_{\flat_1}$. From this, we have
 \begin{align*}
 &\|F(t+\sigma,x(t+\sigma),x(t+\sigma-\tau(t+\sigma)))+L_{\flat_1}-(F(t,x(t),x(t-\tau(t)))+L_{\flat_1})\|_{B^{\flat}}\\
 =&\ \|F(t+\sigma,x(t+\sigma),x(t+\sigma-\tau(t+\sigma)))-F(t,x(t),x(t-\tau(t)))\|_{M_{\flat}}<\varepsilon.
 \end{align*}
 Further, we have
\begin{align*}
 &\|U(t+\sigma)-U(t)\|_{M_{\flat}}\\
 =&\left\|\int_{-\infty}^{t+\sigma}T(t+\sigma-s)F(s,x(s),x(s-\tau(s)))ds
 -\int_{-\infty}^{t}T(t-s)F(s,x(s),x(s-\tau(s)))\right\|_{M_{\flat}}\\
 =&\bigg\|\int_{-\infty}^{t}T(t-s)F(s+\sigma,x(s+\sigma),x(s+\sigma-\tau(s+\sigma)))ds\\
& -\int_{-\infty}^{t}T(t-s)F(s,x(s),x(s-\tau(s)))ds\bigg\|_{M_{\flat}}\\
\leq&\bigg\|\int_{-\infty}^{t}T(t-s)[F(s+\sigma,x(s+\sigma),x(s+\sigma-\tau(s+\sigma)))ds
 -F(s,x(s),x(s-\tau(s)))]ds\bigg\|_{M_{\flat}}\\
\leq&\varepsilon \int_{-\infty}^{t}Ne^{-\lambda(t-s)}ds\\
=&\frac{ \varepsilon N }{\lambda},
 \end{align*}
 thus
 $$\|U(t+\sigma)+L_{\flat_1}-(U(t)+L_{\flat_1})\|_{\mathcal{M}_{\flat}}<\frac{ \varepsilon N }{\lambda},$$
 which means that $U+L_{\flat_1}$ possesses   Bohr's property. When $\sigma$ is replaced by a small enough real number, we know that $U_p+L_{\flat_1}$ is uniformly continuous. Hence, $U+L_{\flat_1}\in B_{AP}^{\flat}(\mathbb{R},\mathbb{H})$. In conclusion, the proof is complete.
\end{proof}

\begin{definition}By a Besicovitch almost periodic mild solution $x=(x_{1},x_{2},\cdots,x_{n})^\mathrm{T}:\mathbb{R}\rightarrow \mathbb{X}$ of system \eqref{eq1}, we mean that $x+L_{\flat_1}\in B_{AP}^{\flat}(\mathbb{R},\mathbb{X})$  and $x$ satisfies
\begin{equation*}
x(t)=\int_{-\infty}^{t}T(t-s)F(s,x(s),x(s-\tau(s)))ds,\,t\in \mathbb{R}.
\end{equation*}
\end{definition}

\begin{theorem}\label{thm41} Suppose that $(H_1)$-$(H_4)$ hold, then  system \eqref{eq1} has a unique Besicovitch  almost periodic mild solution in $\mathbb{W}$.
\end{theorem}
\begin{proof}Define an operator $\Psi:\mathbb{W}\rightarrow \mathbb{W}$ by
\begin{equation*}
(\Psi x)(t)=\int_{-\infty}^{t}T(t-s)F(s,x(s),x(s-\tau(s)))ds,\,\, x\in \mathbb{W}, \,\,t\in \mathbb{R}.
\end{equation*}
 Obviously, it is well-defined and $\Psi$  maps  $\mathbb{W}$ into $\mathbb{W}$ according to Lemma \ref{lem43}. We just have to show that $\Psi: \mathbb{W}\rightarrow \mathbb{W}$ is a contraction mapping. In fact, for any  $x,y\in \mathbb{W}$,
\begin{align*}
\|(\Psi x)(t)-(\Psi y)(t)\|_\mathbb{X}=&\Big\|\int_{-\infty}^{t}T(t-s)\big(F(s,x(s),x(s-\tau(s)))-F(s,y(s),y(s-\tau(s)))\big)ds\Big\|_\mathbb{X}\\
\leq&\int_{-\infty}^{t}Ne^{-\lambda(t-s)}\big(\mathfrak{L}_1\|x(s)-y(s)\|_\mathbb{X} +\mathfrak{L}_2\|x(s-\tau(s))-y(s-\tau(s))\|_\mathbb{X}\big)ds\\
\leq&\frac{N(\mathfrak{L}_1+\mathfrak{L}_2)}{\lambda}\|x-y\|_{\mathbb{W}}, \,\,t\in\mathbb{R},
\end{align*}
 which combined with $(H_4)$ yields
$$\|\Psi x-\Psi y\|_{\mathbb{W}}\leq \kappa\|x-y\|_{\mathbb{W}}.$$
 Hence, $\Psi$ is a contraction mapping from $\mathbb{Y}$ to $\mathbb{W}$. Noting the fact that $(\mathbb{W},\|\cdot\|_{\mathbb{W}})$ is a Banach space, therefore, according to the Banach fixed point theorem, $\Psi$ has a unique fixed point $z^{*}\in \mathbb{W}$ such that $Tz^{*}=z^{*}$. that is, system \eqref{eq1} has a unique Besicovitch almost periodic mild solution $z^{*}$. The proof is complete.
\end{proof}

Now, we give one example to show the effectiveness of the results  obtained in this section.

\begin{example}

\begin{equation}\label{ex1}\left\{\begin{array}{lll}
\frac{\partial}{\partial t}u(t,x)=\frac{\partial^2}{\partial x^2}u(t,x)+ f(t,u(t,x),u(t-\tau(t),x)),\,\, t\geq 0, x\in (0,\pi),\\
u(t,0)=u(t,\pi)=0,\,\, t\geq 0,\\
u(\theta,x)=\varphi(\theta,x), \,\, \theta\in [-\bar{\tau},0], \,\, x\in [0,\pi],
\end{array}\right.
\end{equation}
where $\varphi\in C([-\bar{\tau},0]\times [0,\pi]), \tau(t)=3-\sin(\sqrt{3}t)$ and
\begin{align*}
&f(t,u(t,x),u(t-\tau(t),x))\\
=&\frac{1}{60}\bigg(\cos t+2\cos\sqrt{5}t+4e^{-|t|}-\frac{3}{1+t^2}\bigg)(\sin u(t,x)+3\sin u(t-\tau(t),x)).
\end{align*}
Take $\mathbb{X}=L^2(0,\pi), (Au)x=u''(x)$ for $x\in [0,\pi]$ and $u\in D(A)=\{u\in C^1[0,\pi]: u'$ is absolutely continuous on $[0,\pi], u''\in \mathbb{X}, u(0)=u(\pi)=0\}$.
As we know,   $A$ generates a $C_0$ semigroup $T(t)$ with the property that $\|T(t)\|\leq e^{-t}$ for $t\geq 0$.

Obviously, $f\in B_{AP}^{\flat}(\mathbb{R},\mathbb{X}^2), \tau\in AP(\mathbb{R},\mathbb{R}^+)$, $N=\lambda=1$, $\mathcal{L}_1=\frac{1}{6}, \mathcal{L}_2=\frac{1}{2}$.
By a simple calculation, we have $\kappa=\frac{2}{3}$. So $(H_1)$-$(H_3)$ are satisfied. Thus, according to Theorem \ref{thm41},  system \eqref{ex1}  has a unique  Besicovitch almost periodic solution.
\end{example}

{}

\end{document}